\documentclass[a4paper,11pt]{amsart}
\usepackage{amsfonts,amsmath,amssymb,amsthm,latexsym,color,epsfig}
\usepackage{url,xspace}
\usepackage{tikz}
\usepackage{fullpage}
\usepackage[utf8]{inputenc}
\usepackage{enumitem}

\newcommand{\cl}[1]{\mbox{\ensuremath{\mathbf{#1}}}\xspace}
\newcommand{\NP}{\cl{NP}}
\newcommand{\la}[1]{\mbox{\sc{#1}}\xspace}

\date{}
    \newcommand{\C}{\mathcal{C}}

    \newcommand{\F}{\mathcal{F}}
    \newcommand{\G}{\mathcal{G}}
    
    \newcommand{\HH}{\mathcal{H}}

\newcommand{\ol}{\overline}

\newtheorem{theorem}{Theorem}
\newtheorem{lemma}[theorem]{Lemma}
\newtheorem{corollary}[theorem]{Corollary}

\newtheorem{definition}[theorem]{Definition}
\newtheorem{proposition}[theorem]{Proposition}

\newtheorem{conjecture}[theorem]{Conjecture}

\theoremstyle{remark}
\newtheorem*{remark}{Remark}

\def\qed{\ifhmode\unskip\nobreak\fi\quad\ifmmode\Box\else$\Box$\fi}

\title{Odd-Sunflowers}

\author{Peter Frankl}
\address{(PF) HUN-REN Alfréd Rényi Institute of Mathematics, Budapest, Hungary.
 Partially supported by ERC Advanced Grant GeoScape.}
\email{peter.frankl@gmail.com}

\author{János Pach}
\address{(JP) HUN-REN Alfréd Rényi Institute of Mathematics, Budapest, Hungary and IST, Kloster\-neuburg, Austria.
Partially supported by ERC Advanced Grant GeoScape and NKFIH (National Research, Development and Innovation Office) grant K-131529.}
\email{pach@cims.nyu.edu}

\author{Dömötör Pálvölgyi}
\address{(DP) ELTE Eötvös Loránd University and HUN-REN Alfréd Rényi Institute of
	Mathematics, Budapest, Hungary.
Partially supported by the ERC Advanced Grant ``ERMiD'' and by the J\'anos Bolyai Research Scholarship of the Hungarian Academy of Sciences, and by the New National Excellence Program \'UNKP-23-5 and by the Thematic Excellence Program TKP2021-NKTA-62 of the National Research, Development and Innovation Office.
}
\email{domotor.palvolgyi@ttk.elte.hu}

\date{}

\begin{document}

\maketitle

\begin{abstract}
Extending the notion of sunflowers, we call a family of at least two sets an \emph{odd-sunflower} if every element of the underlying set is contained in an odd number of sets or in none of them. It follows from the Erd\H os--Szemer\'edi conjecture, recently proved by Naslund and Sawin, that there is a constant $\mu<2$ such that every family of subsets of an $n$-element set that contains no odd-sunflower consists of at most $\mu^n$ sets. We construct such families of size at least $1.5021^n$. We also characterize minimal odd-sunflowers of triples.
\end{abstract}

\section{Introduction}\label{intro}\let\thefootnote\relax\footnotetext{A preliminary version of this note already appeared in the proceedings of EuroComb 2023.}

A family of at least three sets is a \emph{sunflower} (or a \emph{$\Delta$-system}) if every element is contained either in all of the sets, or in at most one.
If a family of sets contains no sets that form a sunflower, it is called \emph{sunflower-free}.
This notion was introduced by Erd\H os and Rado~\cite{ErR60} in 1960, and it has become one of the standard tools in extremal combinatorics~\cite{FrT18}. Erd\H os and Rado conjectured that the maximum size of any sunflower-free family of $k$-element sets is at most $c^k$, for a suitable constant $c>0$. This conjecture is still open; for recent progress, see~\cite{AlLWZ21}.
\smallskip

Erd\H os and Szemer\'edi~\cite{ErSz78} studied the maximum possible size of a
sunflower-free family of subsets of  $\{1,\dots,n\}$.
Denote this quantity by $f(n)$ and let $\mu=\lim f(n)^{1/n}$.
Erd\H os and Szemer\'edi conjectured that $\mu<2$, and this was proved by Naslund and Sawin~\cite{NaS17}, using the methods of Croot, Lev, P. Pach~\cite{CrLP17}, Ellenberg and Gijswijt~\cite{ElG17}, and Tao~\cite{Ta16}.
They showed that $\mu< 1.89$, while the best currently known lower bound, $\mu> 1.551$, follows from a construction of Deuber \emph{et al.}~\cite{DeEGKM97}.
\smallskip

Erd\H os, Milner and Rado~\cite{ErMR74} called a family of at least three sets a \emph{weak sunflower} if the intersection of any pair of them has the same size. For a survey, see Kostochka~\cite{Ko00}. In the literature, we can also find pseudo-sunflowers \cite{Fr22} and near-sunflowers \cite{AlH20}. By restricting the parities of the sets, other interesting questions can be asked, some of which can be answered by the so-called linear algebra method (even-town, odd-town theorems; see \cite{BaF92}).
\smallskip

We introduce the following new variants of sunflowers.

\begin{definition}
	A nonempty family of nonempty sets forms an \emph{even-sunflower} (short for even-degree sunflower), if every element of the underlying set is contained in an even number of sets (or in none).

Analogously, a family of at least two nonempty sets forms an \emph{odd-sunflower} (short for odd-degree sunflower), if every element of the underlying set is contained in an odd number of sets, or in none.
\end{definition}

Note that any family of pairwise disjoint sets is an odd-sunflower, but not an even-sunflower.
A (classical) sunflower is an odd-sunflower if and only if it consists of an odd number of sets.
In particular, an odd-sunflower-free family is also sunflower-free, as any sunflower contains a sunflower that consists of three sets.
On the other hand, there exist many odd-sunflowers that contain no sunflower of size three. For example, $\{\{1,2\},\{1,3\},\{2,3\},\{1,2,3\}\}$ is a minimal odd-sunflower.
This example can be generalized as follows.
\smallskip

Let $\C_n$ denote the $(n-1)$-uniform family consisting of all $(n-1)$-element subsets of $\{1,\dots,n\}$. (In some papers this family is denoted by $\binom{[n]}{n-1}$.) Let $\C_n^+$ denote the same family completed with the set $\{1,\dots,n\}$. Obviously, $\C_n$ is an odd-sunflower if and only if $n$ is even, and it is an even-sunflower if and only if $n$ is odd. The family $\C_n^+$ is an odd-sunflower if and only if $n$ is odd, and it is an even-sunflower if and only if $n$ is even. Notice that in any subfamily of these families, the nonzero degrees of the elements differ by at most one. Therefore, in every subfamily of $\C_n$ and $\C_n^+$ which is an odd- or even-suflower, all nonzero degrees need to be the same, showing that $\C_n$ and $\C_n^+$  are minimal odd- or even-sunflowers.
There are many other examples; e.g., all graphs in which every degree is odd/even are 2-uniform odd/even-sunflowers.
In fact, every cycle is a minimal 2-uniform even-sunflower.
In general, it is not hard to show that it is \NP-complete to decide whether an input family is odd-sunflower-free or not (see Appendix \ref{sec:NP}), so there is no hope of a characterization of minimal odd-sunflowers either.
This is in contrast with (classic) sunflowers, where the problem is trivially in P.
Nevertheless, for any fixed $k$, there is a constant number of minimal $k$-uniform odd-sunflowers; we study these in Section \ref{sec:MOS}.

\smallskip

The main question studied in this paper is the following: What is the maximum size of a family $\F$ of subsets of $\{1,\dots,n\}$ that contains no even-sunflower (or no odd-sunflower, respectively)?
We denote these maximums by $f_{even}(n)$ and by $f_{odd}(n)$, respectively.
As in the case of the even-town and odd-town theorems, the answers to these questions are quite different.

\begin{theorem}\label{thm:even}
	For any even-sunflower-free family $\F\subset 2^{\{1,\dots,n\}}$, we have $|\F|\le n$. That is,
	$$f_{even}(n)=n.$$
\end{theorem}

\begin{theorem}\label{thm:odd}
	For every sufficiently large $n$, there exists an odd-sunflower-free family $\F\subset 2^{\{1,\dots,n\}}$ with $|\F|> 1.502148^n$.
	That is, for every  $n>n_0$
	$$f_{odd}(n) > 1.502148^n.$$
\end{theorem}

Let $\mu_{odd}=\lim f_{odd}(n)^{1/n}$.
(The existence of the limit easily follows from our Lemma \ref{lem:sum} and Fekete's lemma, just like for ordinary sunflowers; see \cite{AbH77}.)
Using the fact that any odd-sunflower-free family $\F$ is also sunflower-free, the result of Naslund and Savin~\cite{NaS17} mentioned above implies that $f_{odd}(n)\le 1.89^n$.
Thus, we have
 $$1.502148<\mu_{odd}\le \mu< 1.89.$$
It would be interesting to decide whether $\mu_{odd}$ is strictly smaller than $\mu$, and to find a direct proof for $\mu_{odd}<2$. Is the new slice rank method required?
\smallskip

The starting point of our approach is a 50 years old idea of Abbott, Hanson, and Sauer \cite{AbHS72} concerning ordinary sunflowers: one can use ``direct sums''  to recursively produce larger constructions from smaller ones; see Lemmas 5 and 6 and the discussion on MathOverflow~\cite{Ma14}.

\smallskip
The rest of this paper is organized as follows. In Section \ref{sec:even}, we prove Theorem \ref{thm:even}.
In Section~\ref{sec:brick}, we show that if $n$ is large enough, then the largest odd-sunflower-free families on the underlying set $\{1,\dots,n\}$ cannot be obtained by using only direct sums in the way (to be) described  in Lemma~\ref{lem:sum}. Building on this, in Section \ref{sec:wreath}, we establish Theorem \ref{thm:odd}.
In Section~\ref{sec:MOS}, we study minimal $k$-uniform odd-sunflower-free families and characterize them for $k\le 3$. The final section contains some remarks and open problems.

\section{Proof of Theorem \ref{thm:even}}\label{sec:even}

The lower bound $f_{even}(n)\ge n$ follows from taking $n$ singleton sets.
For the upper bound $f_{even}(n)\le n$, we sketch the argument in two different forms: using linear algebra (as in the usual proof of the odd-town theorem) and by a parity argument (which does not work there).

\begin{proof}[First proof]
	Represent each set by its characteristic vector over $\mathbb F_2^n$.
	If $|\F|>n$, these vectors have a nontrivial linear combination that gives zero.
	The sets whose coefficients are one in this combination yield an even-sunflower.
\end{proof}

\begin{proof}[Second proof.]
	There are $2^{|\F|}-1$ nonempty subfamilies of $\F$.
	If $|\F|>n$, then by the pigeonhole principle, there are two different subfamilies, $\F_1,\F_2\subset \F$, that contain an odd number of times precisely the same elements of $\{1,\dots,n\}$. That is, for every $i\in \{1,\dots,n\}$, we have $$|\{F_1\in\F_1: i\in F_1\}|\equiv |\{F_2\in\F_2: i\in F_2\}| \pmod 2.$$
	But then their symmetric difference, $\F_1\Delta\F_2$, is an even-sunflower.
\end{proof}

\section{Direct Sum Constructions}\label{sec:brick}

Before we prove Theorem \ref{thm:odd}, we give some definitions and state some simple lemmas.

In a \emph{multifamily} of sets, every set $F$ can occur a positive integer number of times. This number is called the \emph{multiplicity} of $F$. A multifamily of at least two nonempty sets is an \emph{odd-sunflower} if the degree of every element of the underlying set is odd or zero.
Note that, similarly to sunflowers, restricting an odd-sunflower multifamily to a smaller underlying set also gives an odd-sunflower multifamily, unless fewer than two nonempty sets remain.

A family $\F$ is called an \emph{antichain}, or \emph{Sperner}, if it is containment-free, i.e., $F,G\in\F$ and $F\subset G$ imply that $F=G$. Let $f_{oa}(n)$ denote the maximum size of an odd-sunflower-free antichain $\F$ on the underlying set $\{1,\dots,n\}$.
Note that any \emph{slice} of $\F$, i.e., any subfamily of $\F$ whose sets are of the same size, form an antichain. Obviously, we have $f_{odd}(n)/n\le f_{oa}(n)\le f_{odd}(n)$ and, therefore, $$\lim f_{oa}(n)^{1/n}=\mu_{odd}.$$

Given two families, $\F$ and $\G$, on different base sets, their \emph{direct sum} is defined as $\F+\G=\{F\cup G\mid F\in \F, G\in \G\}.$
When we write $\F+\F$, then we mean that we take two copies of $\F$ on two disjoint base sets.
We can repeatedly apply this operation to obtain $\F+\F+\dots+\F$.
In such \emph{direct sum constructions}, we call $\F$ the ``building block.''

\smallskip

We start with the following simple construction.

\medskip
\textbf{Construction 1:}
Let $k=\lfloor n/3\rfloor$.
Make $k$ disjoint groups of size 3 from $\{1,\dots,n\}$.
Define $\F$ as the family of all sets that intersect each group in exactly 2 elements.
Then we have $|\F|=3^k$, i.e., $\sqrt[3]{3}^n$, whenever $n$ is divisible by 3.
We prove that this construction is odd-sunflower-free using a series of lemmas, implying
\begin{equation}\label{eq1}
\mu_{odd}\ge \sqrt[3]{3}>1.44.
\end{equation}

\begin{lemma}\label{lem:multi}
	If $\F$ is an odd-sunflower-free family, and $\HH$ is a multifamily of size at least two, comprised of elements $\F$, then $\HH$ is an odd-sunflower multifamily if and only if it consists of an odd number of copies of a single member $F\in \F$, and an even number of copies of some subsets of $F$.
	
	In particular, if $|\HH|$ is even, it cannot be an odd-sunflower.
\end{lemma}

\noindent\emph{Remark.} If $\F$ is an odd-sunflower-free \emph{antichain}, 
then the multifamily $\HH$ is an odd-sunflower if and only if it consists of an odd number of copies of the same set $F\in\F$.
\begin{proof}
The ``if'' part of the statement is obvious.

Assume that $\HH$ is an odd-sunflower.
Reduce the multifamily $\HH$ to a family ${\HH}'$ by deleting all sets of even multiplicity and keeping only one copy of each set whose multiplicity is odd. This does not change the parity of the degree of any element.

Suppose that ${\HH}'\subseteq \F$ consists of at least two sets. Since ${\HH}'\subseteq \F$ is odd-sunflower-free, there is an element which is contained in a nonzero even number of sets of ${\HH}'$ and, therefore, in a nonzero even number of sets in the multifamily $\HH$. This contradicts our assumption that $\HH$ was an odd-sunflower.

If ${\HH}'$ is empty, then any element covered by $\HH$ is contained in an even number of sets from ${\HH}'$, thus $\HH$ again cannot be an odd-sunflower.

Finally, consider the case when the reduced family ${\HH}'$ consists of a single set $F\in \F$.
If all sets in the multifamily $\HH$ are copies of $F$, we are done.
Otherwise, there are some other sets $F'\neq F$ participating in $\HH$ with even multiplicity.
If any such $F'$ has an element that does not belong to $F$, then this element is covered by a nonzero even number of sets of the multifamily $\HH$, contradicting the assumption that $\HH$ is an odd-sunflower.
Therefore, all such $F'$ are subsets of $F$, as claimed.
\end{proof}

\begin{lemma}\label{lem:sum}
	If $\F$ and $\G$ are odd-sunflower-free families, and at least one of them is an antichain, then $\F+\G$ is also odd-sunflower-free.
	Moreover, if both $\F$ and $\G$ are antichains, then so is $\F+\G$.
\end{lemma}

\begin{remark}
	If none of $\F$ and $\G$ are antichains, then it can happen that $\F+\G$ contains an odd-sunflower.
	For example, if $\F=\{\{1\},\{1,2\}\}$ and $\G=\{\{3\},\{3,4\}\}$, then $\{\{1,3\},\{1,2,3\},\{1,3,4\}\}$ is an odd-sunflower.
\end{remark}

\begin{proof} The ``moreover'' part of the statement, according to which $\F+\G$ is an antichain, is trivial.

Suppose for contradiction that $\F + \G$ has a subfamily $\HH$ consisting of at least two sets that form an odd-sunflower.
Without loss of generality, $\G$ is an antichain.

Assume first that the parts of the sets of $\HH$ that come from $\G$ are not all the same. These parts are the restriction of $\HH$ to the underlying set of $\G$, so they form a multifamily which is an odd-sunflower. Applying Lemma \ref{lem:multi} to this subfamily, it follows that the parts of the sets in $\HH$ that come from $\G$ all coincide, contradicting our assumption.

Otherwise, the parts of the sets of $\HH$ that come from $\G$ are all the same, in which case the parts that come from $\F$ are all different.
But then we can use that $\F$ is sunflower-free.
\end{proof}

\begin{corollary}\label{cor:sum} 
	For any integers $n,m,t>0$, we have
	$f_{oa}(n+m)\ge f_{oa}(n)\cdot f_{oa}(m)$, and thus
	$f_{oa}(tn)\ge f_{oa}^t(n)$ and $\mu_{odd}\ge f_{oa}(n)^{1/n}.$
\end{corollary}

This follows by repeated application of Lemma \ref{lem:sum} to the direct sum construction with building block $\F$, i.e., to $\F+\F+\dots+\F$. 
When $\F=\C_3$ consists of the two-element subsets of $\{1,2,3\}$, we recover Construction 1.
This proves (\ref{eq1}).

\section{Wreath Product Constructions}\label{sec:wreath}
In this section, we describe another construction that uses the \emph{wreath product} of two families. This is a common notion in group theory \cite{Hu67}, but less common in set theory. It was introduced in the PhD thesis of the first author \cite{Fr77}; see also \cite{Ma14}.

Let $n, m$ be positive integers, $\F\subseteq 2^{\{1,\dots,n\}}$, $\G\subseteq 2^{\{1,\dots,m\}}$ families of subsets of $N=\{1,\dots,n\}$ and $M=\{1,\dots,m\}$, respectively. Take $n$ isomorphic copies $\G_1,\dots,\G_n$ of $\G$ with pairwise disjoint underlying sets $M_1,\dots, M_n$.
Define the \emph{wreath product} of $\F$ and $\G$, denoted by $\F \wr \G$, on the underlying set $\bigcup_{i=1}^n M_i$, as follows.
\[
\F\wr\G=\{\bigcup_{i\in F} G_i\mid F\in\F, G_i\in\G_i\}.
\]
That is, for each $F\in\F$, for every choice of $G_i\in\G_i$ for every $i\in F$, we take the set $\cup_{i\in F} G_i$.
We obviously have $|\F\wr\G|=
\sum_{F\in\F} |\G|^{|F|}$. Thus, $|\F\wr\G|=|\F||\G|^{k}$ holds, provided that $\F$ is \emph{$k$-uniform}, i.e., $|F|=k$ for every $F\in\F$.


\begin{lemma}\label{lem:wreath}
	If $\F$ and $\G$ are odd-sunflower-free families and $\G$ is an antichain, then $\F\wr\G$ is also odd-sunflower-free.
	Moreover, if $\F$ is also an antichain, then so is $\F\wr\G$.
\end{lemma}
\begin{remark}
	If $\G$ is not an antichain, then it may happen that $\F\wr\G$ contains an odd-sunflower for odd-sunflower-free $\F$ and $\G$, even if $\F$ was an antichain.
	For example, if $\F=\{\{1,2\}\}$ and $\G=\{\{3\},\{3,4\}\}$, then the three sets $\{3_1,3_2\}$,$\{3_1,3_2,4_1\}$,$\{3_1,3_2,4_2\}$ form an odd-sunflower.
\end{remark}
\begin{proof}
	The ``moreover'' part of the statement, according to which $\F\wr\G$ is an antichain, is trivial.

We need to show that in any family $\HH$ of at least two sets from $\F\wr\G$, there is an element contained in a nonzero even number of sets from $\HH$. Consider the multifamily $\HH'$ of sets from $\F$, in which the multiplicity of a set $F$ is as large, as many sets of the form $\cup_{i\in F} G_i$ belong to $\HH$.

Since $\F$ is sunflower-free, there are two possibilities.	
\smallskip

\noindent \emph{Case A}: 
Some set in the multifamily $\HH'$ has multiplicity greater than one.

In this case there exists an element $i\in F$ such that the multifamily of sets from $\G_i$, consisting of the intersections of the sets from $\HH$ with $M_i$, has at least two \emph{distinct} sets.
Otherwise, the sets of $\HH$ that correspond to the repeated set of $\HH'$ would coincide, and $\HH$ has no repeated sets.
Applying Lemma \ref{lem:multi} to the multifamily of sets from $\G_i$ for such an $i$, we find an element of $M_i$ contained in a nonzero even number of sets from $\HH$, as required.
\smallskip

\noindent \emph{Case B}: The multifamily $\HH'$ is not an odd-sunflower. That is, there exists an element $i\in \{1,\dots,n\}$ which is covered by an even number of sets in $\HH'$.

This means that $\HH$ has a nonzero even number of sets with nonempty intersections with $M_i$. Thus, applying Lemma \ref{lem:multi} to the multifamily of sets from $\G_i$ formed by these nonempty intersections, again we find an element of $M_i$ contained in a nonzero even number of sets from $\HH$.

This completes the proof.	
\end{proof}

\begin{corollary}\label{cor:wreath}
Let $\F$ is a $k$-uniform odd-sunflower-free antichain on $n$ elements. Then we have
	$$f_{oa}(nm)\ge |\F|(f_{oa}(m))^k.$$
	In particular, $f_{oa}(nm)\ge n(f_{oa}(m))^{n-1},$ for odd $n$.
\end{corollary}

The second part of the corollary follows by choosing $\F=\C_n$, the family of all $(n-1)$-element subsets of $\{1,\dots,n\}$.
These families have high uniformity, so they are natural candidates to increase the size of the family fast, because the uniformity $k$ appears in the exponent in Corollary \ref{cor:wreath}.

As a simple, concrete application, consider the following.

\bigskip
\textbf{Construction 2:}  
The family $\C_9\wr\C_3$ consists of $|\C_9||\C_3|^8=9\cdot 3^8=3^{10}$ subsets of a $9\cdot 3=27$-element set. Thus, we have
\begin{equation}\label{eq2}
\mu_{odd}\ge |\C_9\wr\C_3|^{1/27}=3^{10/27}>1.502144.
\end{equation}

Lemma~\ref{lem:wreath} implies that $\C_9\wr\C_3$ contains no odd-sunflower.
Thus, $f_{oa}(27)\ge 3^{10}$, and by Corollary \ref{cor:sum}, $\mu_{odd}\ge f_{oa}(27)^{1/27}$.


\bigskip

By Corollaries~\ref{cor:sum} and~\ref{cor:wreath}, we get $\mu_{odd}\ge f_{oa}(mn)^{1/mn}\ge (n|\G|^{n-1})^{1/mn}$. Here, to get the best bound, we need to choose $n$ so as to maximize the last expression. Letting $n=x|\G|$, we obtain
$$\mu_{odd}\ge(n|\G|^{n-1})^{1/mn}=(x|\G|^{n})^{1/mn}=|\G|^{1/m}x^{1/xm|\G|}.$$
Since $|\G|$ and $m$ are independent of $n$, this is equivalent to maximizing $x^{1/x}$.
A simple derivation shows that the optimal choice is $x=e$, so we need $n$ to be the largest odd integer smaller than $e|\G|$, or the smallest odd integer greater than $e|\G|$.
In the case of Construction 2, $3e$ is closest to $9$.

\medskip
The above reasoning also shows that any lower bound $|\G|^{1/m}\le \mu_{odd}$ that comes from the direct sum construction using $\G$ as a the building block, can be slightly improved by taking $\C_n\wr\G$ 
for some odd $n$ close to $e|\G|$.
For example, if $\G=\C_9\wr\C_3$ is the 16-uniform family of $3^{10}$ sets on 27 elements obtained in Construction 2, then we can choose $n$ to be $160511\approx e3^{10}$.

\bigskip
\textbf{Construction 3:}  
The family $\C_{160511}\wr(\C_9\wr\C_3)$ consists of $|\C_{160511}||\C_9\wr\C_3|^{160510}=160511\cdot 3^{1605100}$ subsets of a $160511\cdot 27=4333797$-element set. Thus, we have
\begin{equation}\label{eq3}
	\mu_{odd}\ge (160511\cdot 3^{1605100})^{1/4333797}>1.502148.
\end{equation}

Of course, the improvement on the lower bound for $\mu_{odd}$ is extremely small as the families grow.

\section{Minimal odd-sunflowers (MOS-s)}\label{sec:MOS}

An odd-sunflower is called {\em minimal}, or a {\em MOS}, if it has no proper subfamily which is an odd-sunflower. A $k$-uniform MOS is a {\em $k$-MOS}.
We start with the following simple observation that will help us characterize all MOS's.

\begin{lemma}\label{lem:mos}
	If the underlying set of a MOS $\F$ has $n$ elements, then $|\F|\le n$.
	
	Moreover, if $\F$ is a $k$-MOS for an even $k$, then $|\F|\le n-1$.
\end{lemma}
\begin{proof}
	Assume that $\F$ is an odd-sunflower.
	If $|\F|> n$, then by Theorem \ref{thm:even}, $\F$ has a subfamily $\F'$ which is an even-sunflower.
	But then $\F\setminus\F'$ is an odd-sunflower contained in $\F$, contradicting the minimality of $\F$.
	
	If $\F$ is a $k$-MOS on $n$ elements where $k$ is even, then the sum of the degrees in every subfamily $\F'\subset \F$ is even. So, there are only $2^{n-1}$ options for the degree sequences of the elements modulo 2, over all subfamilies $\F'$.
	If $|\F|=n$, then there are two distinct subfamilies $\F_1,\F_2\subset \F$ that give the same degree sequence. In this case, however, $\F\setminus (\F_1\Delta\F_2)$ would be a smaller odd-sunflower contained in $\F$.
\end{proof}

Up to isomorphism, there is only one 1-MOS: the family consisting of two singletons, $\{\{1\},\{2\}\}$.

We have two different 2-MOS-s: $\{\{1,2\},\{3,4\}\}$, and $\{\{1,2\},\{1,3\},\{1,4\}\}$.
Indeed, if a 2-uniform family does not have the second configuration, then it corresponds to a graph where the maximum degree is two.
In an odd-sunflower every degree is odd. Hence, every degree must be one, in which case we have a collection of disjoint sets: the first configuration.
\smallskip

Note that the above examples either consist of two disjoint sets or form a (classic) sunflower. For 3-MOS-s, this is not true. Notice that if we only considered minimal odd-sunflowers consisting of {\em at least three sets}, then in the 1-uniform case the only minimal example would be $\{\{1\},\{2\},\{3\}\}$, while in the 2-uniform case the examples would be  $\{\{1,2\},\{3,4\},\{5,6\}\}$ and $\{\{1,2\},\{1,3\},\{1,4\}\}$. All of these examples are sunflowers with three petals.
\smallskip

Next, we characterize all 3-MOS-s. Of course, two disjoint 3-element sets form a 3-MOS.
To simplify the notation, in the sequel, we omit the inner set symbols, so this example will be denoted as $\{123, 456\}$.

If in a 3-MOS, there is an element contained in all sets, then deleting these elements gives a 1- or 2-uniform family, which we have characterized already. Because of the odd-degree condition for the elements contained in all sets, we can use only those examples from above that consist of three sets.
These give the following 3-MOS-s:
$\{123,124,125\}$ or $\{123,145,167\}$.

If there is no element contained in all sets of a 3-MOS $\F$, we define for any element $x$ of the underlying set, a graph $G_x$ as follows. Let
$$\F_x=\{F\in\F : x\in F\} \;\; {\rm and}\;\; \F_{\ol x}=\{F\in\F : x\notin F\}.$$
Let the vertices of $G_x$ be the elements of the underlying set of $\F_x$, apart from $x$, and let the edge set of $G_x$ be $E_x=E(G_x)=\{F\setminus\{x\}:F\in\F_x\}$.
From our earlier observations, we can conclude that $G_x$ has maximum degree two and it has no three disjoint edges.
This implies that $|E_x|\le 6$.
\smallskip

However, we can prove a better bound.
First of all, since every degree is odd, $deg(x)=|E_x|$ is also odd, thus $|E_x|\le 5$.
If $|E_x|=5$, then, using the fact that $G_x$ has maximum degree two and it has no three disjoint edges, it must be either a cycle of length five or the disjoint union of a triangle and a path of length two.
We will show that in fact none of these cases is feasible.
\smallskip

If $G_x$ is a cycle on five vertices, then each set in $\F_{\ol x}$ must intersect all edges of $G_x$, because $\F$ is an intersecting family. This implies that the underlying set of $\F$ has only six elements.
Then, by Lemma \ref{lem:mos}, $\F_{\ol x}$ consists of only one set, and it cannot turn the degrees of all five even-degree vertices of $\F_x$ odd.
\smallskip

If $G_x$ is the disjoint union of a triangle $\{12,23,31\}$ and a path $\{45,56\}$, then let $$\F_x=\{x12,x23,x31,x45,x56\}.$$
As $\F$ is intersecting, all sets from $\F_{\ol x}$ must consist of the element $5$, and two of the elements $1,2,3$.
Hence, the degree of one of $1,2,3$ will be even, irrespective of the cardinality of $\F_{\ol x}$, which is impossible.
\smallskip

We obtained

\begin{lemma}\label{lem:mos2}
	In a 3-MOS, in which no element is contained in all sets, the degree of every element $x$ contained in more than one set is $|E_x|=3$.
\end{lemma}

From here, we can conclude that $|\F|\le 7$ as follows.

Pick any set $123\in \F$.
Every set in $\F$ must intersect $123$, so each of them must contribute to at least one of $E_1$, $E_2$, and $E_3$, where $123$ contributes thrice. Therefore, the number of sets in $\F$ is at most $1+3\cdot 2=7$.
Moreover, unless $\F$ is 3-regular, we even have $|\F|\le 5$, by repeating the above argument by picking 1 to be an element included in only one set.
From here, a case analysis (which can be found in Appendix \ref{sec:3MOS}) gives the following.

\begin{proposition}\label{prop:3MOS}
	Up to isomorphisms, we have the following $7$ different 3-MOS-s:
	\begin{enumerate}[label=Case (\arabic*):, leftmargin=6em]
		\item Two disjoint triples: $\{123, 456\}$.
		\item Sunflower of 3 triples with one common element: $\{123,145,167\}$.
		\item Sunflower of 3 triples with two common elements: $\{123,124,125\}$.
		\item $\C_4$: $\{123, 124, 134, 234\}$.
		\item Complement of a $5$-cycle: $\{123, 124, 135, 245, 345\}$. 
		\item $\C_4$ with one element split into three: $\{123,124,135,236\}$.
		\item $\{123,124,156,256,345,346\}$.
	\end{enumerate}
\end{proposition}

Note that each 3-MOS satisfies $|\F|\le 6$.

\begin{remark}
	More generally, with the above reasoning we can bound the number of $k$-tuples in a $k$-MOS as $1+(k-1)g(k)$ where $g(k)$ is the size of the largest sunflower-free family.
	If the Erdős--Rado conjecture is true, this gives an upper bound of $c^k$.
	Maybe it is possible to prove such an exponential bound without invoking the conjecture as well.
	For example, using Lemma \ref{lem:mos} another upper bound is $1+(k-1)g_v(k)$ where $g_v(k)$ is the size of the base set of the largest sunflower-free family.
	We could not find any papers studying the quantity $g_v(k)$, and the base set of most sunflower-free constructions grows only linearly in $k$.
	However, it is not hard to see that we have an exponential lower bound even in the case when the family is odd-sunflower-free: $g_v(k)\ge 2^k-1$.
	This is achieved by the $k$-uniform family whose sets are the root-to-leaf paths in a rooted binary tree of depth $k$.
	Note that this construction is not optimal, in fact, not even maximal; we can add, say, a set that contains the two children of the root, and $k-2$ new vertices.
\end{remark}

\section{Concluding remarks}
In this note, we studied the Erd\H os--Szemerédi-type sunflower problem for odd-sunflowers.
We want to remark that our structural result is also true for (ordinary) sunflowers, using essentially the same proof.

\begin{proposition}
	If $\F$ is any sunflower-free $k$-uniform family on $n$ elements, denoting the direct sum construction with building block $\F$ by $\F^{(t)}=\F+\dots+\F$, then $\lim_{t\to \infty} |\F^{(t)}|^{1/tn}<\mu$.
\end{proposition}

In other words, direct sum constructions will never reach the optimal value $\mu$.
As far as we know, this result is new.
The best currently known examples of Deuber \emph{et al.}~\cite{DeEGKM97} use a combination of a direct sum construction and some other \emph{ad hoc} tricks that do not work for odd-sunflowers.

What about the Erd\H os--Rado-type sunflower problem, i.e., what is the maximum possible size of an odd-sunflower-free $k$-uniform set system?
We pose the following weakening of Erd\H os and Rado's conjecture.

\begin{conjecture}
	The maximum size of any odd-sunflower-free $k$-uniform family is at most $c^k$, for a suitable constant $c>0$.
\end{conjecture}

Note that the respective problem does not make sense for even-sunflowers, as any number of disjoint sets is even-sunflower-free.

\subsubsection*{Added after the first version.}
In the first preprint of this manuscript, we posed the following conjecture, already discussed at the end of Section \ref{sec:MOS}.

\begin{conjecture}\label{conj:solved}
	The maximum number of base elements, each of which is contained in at least one set of a sunflower-free $k$-uniform family, is at most $c^k$, for a suitable constant $c>0$.
\end{conjecture}

It was pointed out by Zach Hunter that a simple argument shows that the maximum number of base elements grows roughly the same way as the maximum number of sets, see \url{https://mathoverflow.net/a/463150/955}.
In particular, his argument shows that Conjecture \ref{conj:solved} is equivalent to the Erdős-Rado conjecture about the size of sunflower-free $k$-uniform families.

\subsubsection*{Acknowledgment}

We are grateful to Bal\'azs Keszegh, and to the members of the Mikl\'os Schweitzer Competition committee of 2022 for valuable discussions, and Shira Zerbib for pointing out several important mathematical typos.

\appendix

\section{\NP-hardness sketch}\label{sec:NP}
\begin{proposition}
	It is \NP-complete to decide whether an input family contains an odd-sunflower or not.
\end{proposition}
We will denote the above decision problem by \la{OddSF}.
\begin{proof}
\la{OddSF} is obviously in \NP, as we can check whether a subfamily is an odd-sunflower.
To prove \NP-hardness, we reduce the \NP-complete problem \la{3DM} (3-Dimensional Matching) to \la{OddSF}.
In the \la{3DM} problem, our input is a 3-partite 3-uniform hypergraph, with $n$ vertices in each part, and our goal is to decide whether there are $n$ edges that cover each vertex exactly once, called a 3-dimensional matching.

For simplicity, we assume that $n$ is odd.
Our goal is to convert a 3-partite 3-uniform hypergraph $\HH$ into a family $\F$ such that $\F$ is odd-sunflower-free if and only if $\HH$ has no 3-dimensional matching.
The base set of $\F$ will be the set of vertices of $\HH$, and some additional elements.
From each edge $e$ of $\HH$, we take $n$ copies for our family $\F$, and to each copy of $e$ we add some further elements.
We take a $\C_n$ on $n$ new elements, and we add a different $(n-1)$-element set from it to each of the $n$ copies of each edge $e$.
Furthermore, for any two copies of some edges $e$ and $f$ that contain the same vertex $v$ of $\HH$, i.e., $v\in e,f$, we take a new element that we add to only these two copies.
Note that we allow $e=f$, so in particular, we take a new element for any two copies of the same edge.
This ensures that in an odd-sunflower the corresponding original edges in $\F$ cover each vertex at most once.
Similarly, for any two copies of different edges $e$ and $f$ that contain the same $(n-1)$-element set from $\C_n$, we take a new element that we add to only these two copies.
This ensures that an odd-sunflower contains each $(n-1)$-element set from $\C_n$ exactly once, as $\C_n$ is a minimal odd-sunflower.
Therefore, any odd-sunflower consists of exactly $n$ sets, each of which is derived from a copy of a different edge, which were pairwise disjoint in $\HH$; in other words, they form a 3-dimensional matching in $\HH$.
This finishes the proof of the \NP-hardness.
\end{proof}

\section{Proof of Proposition \ref{prop:3MOS}}\label{sec:3MOS}

We prove that according to which up to isomorphisms, we have the following 7 different 3-MOS-s:
\begin{enumerate}[label=Case (\arabic*):, leftmargin=6em]
	\item Two disjoint triples: $\{123, 456\}$.
	\item Sunflower of 3 triples with one common element: $\{123,145,167\}$.
	\item Sunflower of 3 triples with two common elements: $\{123,124,125\}$.
	\item $\C_4$: $\{123, 124, 134, 234\}$.
	\item Complement of a $5$-cycle: $\{123, 124, 135, 245, 345\}$. 
	\item $\C_4$ with one element split into three: $\{123,124,135,236\}$.
	\item $\{123,124,156,256,345,346\}$.
\end{enumerate}

\begin{proof}
	Assume that $\F$ is a 3-MOS.
	$\F$ is intersecting unless $\F$ consists of two disjoint triples, case (1).
	
	If $\F$ is linear (any two sets intersect in exactly one element), then take any $\{x\}=F_1\cap F_2$.
	Since the degree of $x$ is odd, there is an $F_3 \ni x$, in which case $F_1,F_2,F_3$ form a 3-sunflower, case (2).
	
	Otherwise, without loss of generality, assume that $123,124\in\F$.
	Using that every degree is either one or three, there is exactly one other set $F$ that contains 1.
	If $1,2\in F$, we are in case (3).
	
	If $F=134$, every set must contain at least two of 2,3,4, otherwise they would be disjoint from 123 or 124.
	But since the degree of each of them is three, this is only possible if there is exactly one more set containing them all, case (4).
	
	More generally, if there are three sets that pairwise intersect in two elements, we are in case (3) or case (4).
	
	If $F=135$, every other set must contain 2 or 3 to intersect 123, so we have either two, or just one more set.
	To satisfy that each degree is odd, in the former case we must be in case (5), in the latter in case (6).
	If $F$ is 145, 235, 245 cases are the same.
	
	Finally, if $F=156$, and we are in none of the previous cases, then to ensure that every other set intersects 123 and 124, there is one set that contains 2, but not 3 or 4, and two sets that both contain 3 and 4.
	Moreover, all these sets must also contain one of 5 and 6 to intersect 156.
	The only option is case (7).
\end{proof}

\end{document}